\documentclass{amsart}
\usepackage{}
\usepackage{color}

\newcommand{\ov}{\overline}

\newcommand{\ol}{\overline}

\newtheorem{theorem}{Theorem}[section]
\newtheorem{lemma}[theorem]{Lemma}

 \theoremstyle{definition}
\newtheorem{definition}[theorem]{Definition}

\theoremstyle{remark}
\newtheorem{remark}[theorem]{Remark}

\numberwithin{equation}{section}

\begin{document}

\title[$C^2$ estimates for $k$-Hessian equations]
{Second order estimates for a class of complex Hessian equations on Hermitian manifolds}

\author{Weisong Dong}
\address{School of Mathematics, Tianjin University,
         Tianjin, P.R.China, 300354}
\email{dr.dong@tju.edu.cn}


\begin{abstract}

In this paper, we derive an \emph{a priori} second order estimate for solutions which are in $\Gamma_{k+1}$ cone
to a class of complex Hessian equations
with both sides of the equation depending on the gradient on compact Hermitian manifolds.



\end{abstract}

\maketitle

\section{Introduction}

Let $(M, \omega)$ be a compact complex manifold of complex dimension $n \geq 2$ with Hermitian metric $\omega$.
For any smooth function $u \in C^{\infty}(M)$, let $\chi(z,u)$ be a smooth real (1,1) form on $M$
and $\psi(z, v, u) \in C^{\infty}\left(\left(T^{1,0}( M ) \right)^{*} \times \mathbb{R}\right)$ be a positive function, where $T^{1,0} (M)$ is the holomorphic tangent bundle.
Given any smooth $(1,0)$-form $a$ on $M$,
we obtain a new real $(1,1)$-form
\begin{equation}
\label{g}
g
=\chi(z, u)+\sqrt{-1} a \wedge \overline{\partial} u - \sqrt{-1} \overline{a} \wedge \partial u+\sqrt{-1} \partial \overline{\partial} u.
\end{equation}
Consider the following complex Hessian equation
\begin{equation}\label{eqn}
Q ( \lambda ) = \sum_{s = 1}^{k} \alpha_s \sigma_s (\lambda ) = \psi(z, Du, u), \;\; \mbox{for}\;\; 1\leq k \leq n,
\end{equation}
where $\lambda = (\lambda_1, \cdots, \lambda_n) \in \Gamma_{k+1}$ are the eigenvalues of $g$ with respect to
$\omega$ and
$\alpha_1, \cdots, \alpha_{k-1}$ are non-negative constants and $\alpha_k$ is a positive constant.
For  $\lambda = \left(\lambda_{1}, \cdots, \lambda_{n}\right) \in \mathbb{R}^{n}$,
$\sigma_s (\lambda)$ is the $s$-th elementary symmetric function defined by
$\sigma_{s}(\lambda)=\sum \lambda_{i_{1}} \lambda_{i_{2}} \cdots \lambda_{i_{s}}$,
where the sum is over $\{1\leq i_{1}<\cdots<i_{s}\leq n\}$.
We will also sometimes use the convention
$\sigma_0 (\lambda) = 1$ and $\sigma_s (\lambda) = 0$ if $s > n$ or $s < 0$.
The cone $\Gamma_{s}$ is defined by
\begin{equation}
\Gamma_{s} :=\left\{\lambda \in \mathbb{R}^{n} : \sigma_{j}(\lambda)>0, j=1, \cdots, s \right\}.
\end{equation}
Apparently, $\Gamma_1 \supset \Gamma_{2} \supset \cdots \supset \Gamma_n$.
It was shown in \cite{C-N-S} that $\Gamma_s$ is an open convex symmetric cone in $\mathbb{R}^n$ with vertex at the origin
and the equation is elliptic if $g \in \Gamma_k (M)$ (see Section 2 for the notation).

If the coefficients $\alpha_1 = \cdots = \alpha_{k-1} = 0$ in \eqref{eqn}, it is just the $k$-Hessian equation, and
the second order estimate
was established by Phong-Picard-Zhang \cite{P-P-Z1} and Dong-Li \cite{DL}.
The $k$-Hessian equations in both real and complex case
are related to many geometric problems and
have been studied extensively.
Equations of a similar form as \eqref{eqn} also emerge frequently.
In the Calibrated geometry,
the special Lagrangian equations introduced by Harvey and
Lawson \cite{HL}
can be written as the alternative combinations of elementary symmetric
functions.
A complex analogue of the special Lagrangian equation
appeared naturally from the study of Mirror Symmetry,
see \cite{LYZ,CJY,CY}.
Another important example
is the Fu-Yau \cite{F-Y2, F-Y1} equation arising from the
study of Hull-Strominger system.
There have been a lot of works on this topic recently,
see \cite{P-P-Z1, P-P-Z2, C-H-Z1, P-P-Z3, C-H-Z2}.
Such type of equations also originate in the study of $J$-equation on toric
varieties by Collins-Sz\'ekelyhidi \cite{CS}.
The real analogous equation of \eqref{eqn} was studied by Li-Ren-Wang \cite{LiRW},
and see Guan-Zhang \cite{GZ} for interesting related work.


A typical example of the complex Hessian equation
involving a linear gradient term on the left hand side
arising from the Gauduchon conjecture was
studied by Sz\'ekelyhidi-Tosatti-Weinkove \cite{S-T-W}.
See also Guan-Nie \cite{GN} for a work on related results.
The Monge-Amp\`ere equation with an
additional linear gradient term inside the determinant can be found in
the study of the Calabi-Yau equation on certain symplectic non-K\"ahler 4-manifolds
by Fino-Li-Salamon-Vezzoni \cite{FinoLi}.
Many progresses have been made recently for general
complex Hessian equations involving a linear gradient term on the left hand sides,
see Tosatti-Weinkove \cite{T-W19}, Feng-Ge-Zheng \cite{Feng-Ge-Zheng}  and Yuan \cite{Yuan,Yuan2}.
In the Fu-Yau equation,
the right hand side function $\psi$ of a particular structure depends on the gradient $Du$.
Phong-Picard-Zhang \cite{P-P-Z1} first investigated the complex $k$-Hessian equation with general right hand side $\psi (z, Du, u)$
on K\"ahler manifolds.
The complex $k$-Hessian equation with gradient terms on both sides
was studied
by Li and the author \cite{DL} on Hermitian manifolds,
and the second order estimate was derived for $g\in \Gamma_n (M)$.
In this paper, we concentrate on the second order estimate for equation \eqref{eqn} with the condition $g \in \Gamma_{k+1}(M)$.
We need to assume that $Q$ satisfies the so called quotient concavity property introduced by Li-Ren-Wang \cite{LiRW}.
\begin{definition}[\cite{LiRW}]
Suppose that $k-1$ polynomials $S_1, \cdots, S_{k-1}$ are defined by
\[
S_l (\lambda) = \sigma_l (\lambda) + \sum_{s=0}^{l-1} \beta_s^l \sigma_s (\lambda), \; \mbox{for}\; 1\leq l \leq k-1,
\]
where $\beta_s^l$ are constants depending on indices $s$ and $l$. If
the function $Q(\lambda) = \sum_{s=1}^k \alpha_s \sigma_s (\lambda)$ satisfies that $(Q/S_l)^{1/(k-l)}$ are concave functions with respect to $\lambda = (\lambda_1, \cdots, \lambda_n)$,
we call that the operator $Q$ is quotient concave.
\end{definition}

We prove the following main result.

\begin{theorem}\label{thm1}
Let $(M,\omega)$ be a compact Hermitian manifold of complex dimension $n$
and $u \in C^{\infty}(M)$. Suppose $g$ defined in \eqref{g} is in $\Gamma_{k+1} (M)$ and satisfies \eqref{eqn}.
Assume $\chi(z, u) \geq \varepsilon \omega$ and $Q$ is quotient concave.
Then we have the uniform second order derivative estimate
\begin{equation}
|D \ov{D} u|_{\omega} \leq C,
\end{equation}
where $C$ is a uniform constant depending only on $(M,\omega)$, $\varepsilon$, $n$, $k$, $\chi$, $\psi$, $a$, $\sup_{M}|u|$, $\sup_{M}|D u|$.
\end{theorem}

\begin{remark}
For a non-negative constant $\alpha$,  $Q = \alpha \sigma_{k-1} (\lambda) + \sigma_k (\lambda)$
satisfies the quotient concavity, see \cite{LiRW}.
In general, $Q$ may not be quotient concave.
A sufficient condition for $Q$ to be quotient concave is given by Theorem 3 in \cite{LiRW}.
\end{remark}

\begin{remark}
The quotient concavity is crucial to our estimate since it will give us more good third order terms
by Lemma \ref{GRW} and $g\in \Gamma_{k+1} (M)$,
see \eqref{mu1} and \eqref{A-2}.
\end{remark}

\begin{remark}
There are no further assumptions on $\psi$ besides that $\psi > 0$.
If $k=n$ in \eqref{eqn}, our method works for $g\in \Gamma_n (M)$ which is a natural elliptic condition.
For $1< k < n$, the more natural assumption to derive the estimate is that $g \in \Gamma_k(M)$, which is still open.
\end{remark}

If the right hand side function $\psi$ does not depend on $Du$,
a second order estimate for complex $k$-Hessian equations on compact K\"ahler manifolds
was first derived by Hou-Ma-Wu \cite{H-M-W}. The particular form of their estimate was used to establish
gradient estimates by a blowup argument and Liouville type theorem due to Dinew-Kolodziej \cite{D-K},
where the complex $k$-Hessian equations were solved.
On Hermitian manifolds, the problem was settled by Sz\'ekelyhidi \cite{G.S} and Zhang \cite{D.K.Zhang}.
See Collins-Picard \cite{CP} for the Dirichlet problem.
We remark that Zhang \cite{X.W.Zhang}
proved a uniform gradient estimate for solutions in $\Gamma_{k+1}$ cone to the complex $k$-Hessian equations
involving gradient terms on the left hand sides.

The function $\psi$ depending on the gradient $Du$ creates substantial new difficulties due to the difference between the terms
$|D D u|^{2}$ and $|D \ov{D} u|^{2}$.
A consequence of this is that we cannot control the bad third order terms directly as in Li \cite{YYLi} or \cite{H-M-W}.
But, in the real case, if $\psi$ is convex with respect to $Du$, one can establish the
second order estimate as Guan-Jiao \cite{G-J}.
Without the convexity assumption on $\psi$,
the desired estimate
was first derived for solutions in $\Gamma_n$ cone to the real $k$-Hessian equations by Guan-Ren-Wang \cite{G-R-W}
and then for solutions in $\Gamma_{k+1}$ cone by Li-Ren-Wang \cite{L-R-W}.
For the real counterpart of equation \eqref{eqn},
the second order estimate was proved by Li-Ren-Wang \cite{LiRW}.

For the complex $k$-Hessian equations,
Phong-Picard-Zhang \cite{P-P-Z} first generalized the result in \cite{G-R-W} to K\"ahler manifolds.
It is more difficult to control the negative third order terms due to complex conjugacy,
since we get only half as many useful terms $``B"$ and $``D"$ (see Section 3) as in the real case.
These are the main difficulties that have been overcome in \cite{P-P-Z}.
On Hermitian manifolds,
there will be more bad third order terms of the form $T*D^3u$, where $T$ is the torsion of $\omega$.
Furthermore, the linear gradient terms in $g$ also bring some bad terms, see $``H"$ and $``I"$ in Section 3.
The method in \cite{P-P-Z} cannot be used to overcome these new difficulties.
Therefore, we apply the maximum principle to the test function used by the author in \cite{DL}.
This gives us a little more good third order terms which are sufficient to push the argument through.

The rest of the paper is organized as follows.
In Section 2, we introduce some useful notations, properties of the $k$-th elementary symmetric functions,
and some preliminary calculations and estimates.
In Section 3, we prove Theorem \ref{thm1}.

\textbf{Acknowledgements}:
 The author is supported by the National Natural Science Foundation of China, No.11801405 and No. 62073236.

\section{Preliminaries}

First, we introduce some notations.
Let $A^{1,1}(M)$ be the space of smooth real $(1,1)$-forms on $(M, \omega)$. For any $h \in A^{1,1}(M)$,
written in local coordinates as $h = \sqrt{-1} h _{i\bar j} d z^i \wedge d z^{\bar j}$,
we say
\[
h \in \Gamma_k (M)
\]
if the vector of eigenvalues of the Hermitian endomorphism ${h^i}_j = \omega^{i\bar k} h_{j \bar k}$ lies in the
$\Gamma_k$ cone at each point.
With the above notation, in local coordinates, \eqref{eqn} can be rewritten as follows:
\begin{equation}\label{Sk2}
Q ( {g^i}_j )= Q \left(\omega^{i\bar k} (\chi_{j \overline{k}} + u_{j \overline{k}}
+ a_{j} u_{\overline{k}} + a_{\overline{k}} u_{j}) \right) = \psi(z, D u, u).
\end{equation}



In local complex coordinates $\left(z_{1}, \ldots, z_{n}\right)$, the subscripts of a function $u$ always denote the covariant derivatives of $u$ with respect to the Chern connection of $\omega$ in the directions of the local frame $\left(\partial / \partial z^{1}, \ldots, \partial / \partial z^{n}\right)$. Namely,
\begin{equation}\nonumber
u_{i}=D_{i}u=D_{\partial / \partial z^{i}} u, \; u_{i \overline{j}}=D_{\partial / \partial \overline{z}^{j}} D_{\partial / \partial z^{i}} u, \; u_{i \overline{j} l}=D_{\partial / \partial z^{l}} D_{\partial / \partial \overline{z}^{j}} D_{\partial / \partial z^{i}} u.
\end{equation}
We have the following commutation formula on Hermitian manifolds (see \cite{T-W3} for more details):
\begin{equation}\label{order}
\begin{aligned}
u_{i \overline{j} \ell}=  u_{i \ell \overline{j}}-u_{p} R_{\ell \overline{j} i}{}^p, \;
u_{p \overline{j} \overline{m}}=  u_{p \overline{m} \overline{j}}-\overline{T_{m j}^{q}} u_{p \overline{q}},\;
u_{i \overline{q} \ell}=  u_{\ell \ov q i}-T_{\ell i}^{p} u_{p \overline{q}},
\end{aligned}
\end{equation}
\begin{equation} \label{order2}
u_{i \overline{j} \ell \overline{m}} =  u_{\ell \overline{m} i \overline{j}}
+u_{p \overline{j}} R_{\ell \overline{m} i}{}^{p}
 - u_{p \overline{m}} R_{i \overline{j} \ell}{}^{p}
- T_{\ell i}^{p} u_{p \overline{m} \overline{j}}
-\overline{T_{m j}^{q}} u_{\ell \overline{q} i}
 -T_{i \ell}^{p} \overline{T_{m j}^{q}} u_{p \overline{q}}.
\end{equation}
As in \cite{CP} and \cite{P-P-Z3}, we define the tensor
\[
\sigma_k^{p\bar q} = \frac{\partial \sigma_k}{\partial {g^r}_p} \omega^{r \bar q}\; \mbox{and}\;
\sigma_k^{p\bar q, r\bar s} = \frac{\partial^2 \sigma_k}{\partial {g^a}_p \partial {g^b}_r} \omega^{a\bar q} \omega^{b\bar s}.
\]
Then, we introduce the following notations
\[
Q^{p\bar q} = \sum_{s=1}^{k} \alpha_s \sigma_s^{p \bar q},  \;\;
Q^{p \overline{q}, r \overline{s}} = \sum_{l=1}^{k} \alpha_l \sigma_l^{p\bar q, r\bar s},\;\;
\mathcal{Q}=\sum_{p} Q^{p \overline{q}} \omega_{\bar q p},
\]
and the following notations as in \cite{P-P-Z}
\begin{equation}
\begin{aligned}
|D D u|_{Q }^{2}= Q^{p \overline{q}} \omega^{m \overline{\ell}} u_{mp} u_{\ov \ell \ov q},\;
|D \ov{D} u|_{Q}^{2}= Q^{p \overline{q}} \omega^{m \overline{\ell}} u_{p\ov\ell} u_{m\ov q},\;
|\eta|_{Q}^{2}= Q^{p \overline{q}} \eta_{p} \eta_{\overline{q}},
\end{aligned}
\end{equation}
for any 1-form $\eta$.

We use $\sigma_k (\lambda|i)$ to denote the $k$-th elementary symmetric function with $\lambda_i = 0$
and $\sigma_k(\lambda|ij)$ the $k$-th elementary function with $\lambda_i =\lambda_j = 0$.
We list here some properties of the $k$-th elementary symmetric function.
\begin{lemma}
\label{sigmak}
For $\lambda = (\lambda_1, \cdots, \lambda_n) \in \mathbb{R}^n$ and $k = 1, \cdots, n$, we have
\begin{itemize}
\item[(1)] $ \sigma_k (\lambda) = \sigma_k(\lambda | i) + \lambda_i \sigma_{k-1} (\lambda | i),\; \forall \; 1 \leq i \leq n;$
\item[(2)] $ \sum_{i=1}^n \sigma_k (\lambda | i) = (n-k) \sigma_k(\lambda);$
\item[(3)] For $\lambda\in \Gamma_k $ and $\lambda_1 \geq \cdots \geq \lambda_n$, we have $\lambda_1 \sigma_{k-1} (\lambda | 1) \geq \frac{k}{n} \sigma_k(\lambda);$
\item[(4)] If $\lambda \in \Gamma_{k}$, we have $\lambda | j \in \Gamma_{k-1}$ for all $1\leq j \leq n$;
\item[(5)] If $\lambda_i \geq \lambda_j$, we have $\sigma_{k-1} (\lambda |i ) \leq \sigma_{k-1} (\lambda | j)$;
\item[(6)] For $\lambda\in \Gamma_k $ and $\lambda_1 \geq \cdots \geq \lambda_n$,
           we have $\sigma_k (\lambda) \leq C \lambda_1 \cdots \lambda_k $.
\end{itemize}
\end{lemma}

\begin{proof}
For (1) (2), it is trivial. For (3) see \cite{H-M-W}. For (4) see \cite{HS}. For (5) and (6), see \cite{YYLi}.
\end{proof}

\begin{lemma}
\label{sigmak-1}
For $\lambda \in \Gamma_{k+1}$ with ordering $\lambda_1 \geq \cdots \geq \lambda_n$,
we have $\sigma_k \geq \lambda_1 \cdots \lambda_k$.
If $\lambda$  satisfies $Q (\lambda )= \psi$ in addition,
we have $\lambda_i + K_0 \geq 0$ for a uniform positive constant $K_0$ depending on $\psi$ and $\alpha_k$.
\end{lemma}
For the proof of the above lemma, see \cite{L-R-W}.
We also need the following lemma.
\begin{lemma}[\cite{G-R-W}]
\label{GRW}
Suppose $1 \leq \ell<k \leq n$, and let $\beta =1 /(k-\ell)$. Let $W=(w_{p\overline{q}})$ be a Hermitian tensor in the $\Gamma_{k}$ cone.
$Q$ given in \eqref{eqn} satisfies the quotient concavity.
Then for any $\theta>0$,
\begin{equation}
\begin{aligned}
&\ - Q^{p \ov{p}, q \ov{q}}(W) w_{p \overline{p} i} w_{q \overline{q} \overline{i}}+\left(1-\beta +\frac{\beta}{\theta}\right) \frac{\left|D_{i} Q(W)\right|^{2}}{Q(W)}\\
\geq &\ Q(W)(\beta + 1- \beta \theta)\left|\frac{D_{i} S_{\ell} (W)}{S_{\ell} (W)}\right|^{2}-\frac{Q}{S_\ell}(W) S_{\ell}^{p \overline{p}, q \overline{q}}(W) w_{p \overline{p} i} w_{q \overline{q} \overline{i}}.
\end{aligned}
\end{equation}
\end{lemma}

The proof is similar to the proof of Lemma 2.2 in \cite{G-R-W} since $(Q/S_\ell)^{1/(k-\ell)}$ are concave functions.

\begin{lemma}[\cite{B}]
\label{B}
 Suppose that $F(A)= f(\lambda_1, \ldots, \lambda_n)$ is a symmetric function of the eigenvalues of a Hermitian matrix $A=({h^i}_j)$, then at a diagonal matrix $A$ with distinct eigenvalues, we have,
\begin{align}
\label{symmetric func 1th deriv} \frac{\partial F}{\partial {h^i}_j} =&\ \delta_{ij} f_i,\\
\label{symmetric func 2th deriv} \frac{\partial^2 F}{\partial {h^i}_j \partial {h^r}_s} {T^i}_j {T^r}_s
 =&\ \sum f_{ij} {T^i}_i {T^j}_j + \sum_{p\neq q}\frac{f_p - f_q}{\lambda_p-\lambda_q} | {T^p}_q |^2,
\end{align}
where $T$ is an arbitrary Hermitian matrix.
\end{lemma}
Note that these formulae make sense even when the eigenvalues are not distinct, since it can be interpreted as a limit.

Now we do some basic calculations which will be used in the next Section.
In the following, $C$ will be a uniform constant depending on the known data as in Theorem \ref{thm1} but may change from line to line.

Our calculations are carried out at a point $p_0$ on the manifold $M$, and we use coordinates such that at this point $\omega=\sqrt{-1} \sum \delta_{k \ell} dz^{k} \wedge d \overline{z}^{\ell}$ and $g_{i \ov{j}}$ is diagonal.
Also, $\lambda_{1}, \lambda_{2}, \dots, \lambda_{n}$ are the eigenvalues of
${g^i}_j$ with the ordering $\lambda_{1} \geq \lambda_{2} \geq \cdots \geq \lambda_{n}$.
Note that $\{Q^{p\ov q}\}$ is diagonal at the point $p_0$ by Lemma \ref{B}.
Moreover, we see
\[
Q^{p\bar p} = \frac{\partial Q}{\partial \lambda_p} = \sum_{s=1}^k \alpha_s \sigma_{s-1} (\lambda | p)\; \mbox{and} \;
Q^{p\bar p, q\bar q} = \frac{\partial^2 Q}{\partial \lambda_p \partial \lambda_q} = \sum_{s=1}^k \alpha_s \sigma_{s-2} (\lambda | pq),
\]
since $\sigma_s^{p\bar p} = \frac{\partial \sigma_s}{\partial \lambda_p} = \sigma_{s-1} (\lambda |p)$
and $\sigma_s^{p\bar p, q\bar q} = \frac{\partial^2 \sigma_s}{\partial \lambda_p \partial \lambda_q} = \sigma_{s-2} (\lambda |pq )$.
Using \eqref{symmetric func 2th deriv}, one can obtain the well-known identity
\begin{equation}
\label{identity}
- Q^{p \overline{q}, r \overline{s}} D_{j} g_{p \overline{q}} D_{\overline{j}} g_{r \overline{s}}
= -Q^{p \overline{p}, q \overline{q}} D_{j} g_{p \overline{p}} D_{\overline{j}} g_{q \overline{q}}
+ Q^{p \overline{p}, q \overline{q}}\left|D_{j} g_{q \overline{p}}\right|^{2}.
\end{equation}

Differentiating \eqref{Sk2} yields
\begin{equation}\label{differential equ}
Q^{p \overline{q}} D_{i} g_{p \overline{q}}=D_{i} \psi.
\end{equation}
Differentiating the equation a second time gives
\begin{equation}\label{diff equ second time}
\begin{split}
& Q^{p \overline{q}} D_{\ov{j}} D_{i} g_{p \ov q}+ Q^{p \overline{q}, r \overline{s}}D_{\overline{j}} g_{r \overline{s}} D_{i} g_{p \overline{q}}=D_{\ov j}D_{i} \psi \\
\geq & -C(1+|DDu|^{2}+|D \ov{D} u|^{2})+\sum_\ell \psi_{v_\ell} u_{\ell i \ov{j}} + \sum_\ell \psi_{\ov{v}_\ell} u_{\ov{\ell} i \ov{j}}\\
\geq & -C(1+|DDu|^{2}+|D \ov{D} u|^{2})+\sum_\ell \psi_{v_\ell} g_{i \ov{j} \ell} + \sum_\ell \psi_{\ov{v}_\ell} g_{i \ov{j} \ov{\ell}}-C \lambda_{1}.
\end{split}
\end{equation}
Direct calculation gives the estimate
\begin{align}\label{4th order term}
Q^{p \ov{q}} D_{\ov{q}} D_p g_{i \ov{j}}
\geq Q^{p \ov{q}} D_{\ov{q}} D_p D_{\ov{j}} D_i u +  Q^{p \ov{q}} D_{\ov{q}} D_p (a_i u_{\ov j} + a_{\ol j} u_{i})
- C \lambda_1 \mathcal{Q}.
\end{align}
By \eqref{order} and \eqref{order2}, commuting derivatives yields that
\begin{equation}
\begin{aligned}\label{commute 4th order deri}
u_{i\bar j p \bar q}
=&\ D_{\ov{j}} D_i g_{p \ov{q}} -D_{\ov{j}} D_i (a_p u_{\ol q} + a_{\ol q} u_p) - D_{\ov{j}} D_i \chi_{p \ov{q}} \\
&\ -  R_{i \ov{j} p}{}^{a} u_{a \ov{q}} + R_{p \ov{q} i}{}^{a} u_{a \ov{j}}
 -T_{p {i}}^{a} u_{a \ov{q} \ov{j}} -\ov{T_{q j}^{a}} u_{p \ov a i}-T_{i p}^{a} \ov{T_{q j}^{b}} u_{a \ov{b}}.
\end{aligned}
\end{equation}
Combining \eqref{diff equ second time}, \eqref{4th order term} and \eqref{commute 4th order deri}, we have
\begin{equation}
\begin{aligned}\label{dif-eqn2}
&\ Q^{p \overline{p}} D_{\overline{p}} D_{p} g_{j \overline{j}}\\
\geq &\ -Q^{p \overline{q}, r \overline{s}} D_{\ov{j}}  g_{r \overline{s}} D_{j} g_{p \overline{p}}+\sum \psi_{v_\ell}
g_{j \ov{j} \ell} + \sum \psi_{\ov{v}_\ell} g_{j \ov{j} \ov{\ell}}- Q^{p \ov{p}} (T_{p j}^{a} u_{a \overline{p} \overline{j}} \\
&\ +\overline{T_{p j}^{a}} u_{p \overline{a} j})-C(1+|DDu|^{2}+|D \ov{D}u|^{2}+{\lambda}_{1} \mathcal{ Q}+ \lambda_{1})\\
&\ + Q^{p \overline{p}}(a_j u_{\ov j p \ov p} + a_{\ov j} u_{j p \ov p} - a_p u_{\ov p j \ov j} - a_{\ov p} u_{pj \ov j})\\
&\ + Q^{p \overline{p}} (a_{jp} u_{\ov j \ov p}+a_{\ov j\ov p}u_{jp}- a_{pj}u_{\ov p \ov j}-a_{\ov p \ov j }u_{pj}).
\end{aligned}
\end{equation}
Using the Cauchy inequality, we have
\begin{equation}
\label{dif-eqn2-1}
\begin{aligned}
 Q^{p \overline{p}} (a_{jp} u_{\ov j \ov p}+a_{\ov j \ol p}u_{jp}-a_{pj}u_{\ov p \ov j}-a_{\ov p \ol j}u_{pj})
\leq \frac{1}{4}|DDu|_{ Q }^{2} + C\mathcal{ Q }.
\end{aligned}
\end{equation}
By \eqref{order}
direct calculation gives
\begin{equation}
\begin{aligned}\label{(Du)_pq 2}
& Q^{p \overline{q}}|D u|_{p \overline{q}}^{2}\\
= &\ Q^{p\bar q} (u_{mp\bar q} D^m u + u_m u_{\bar \ell p \bar q} \omega^{m \bar\ell})
+ |D D u|_{Q }^{2}+|D \ov{D} u|_{Q }^{2}\\
=&\ Q^{p \overline{q}} D_{m}\left( g_{p \overline{q}}-\chi_{p \overline{q}}\right) D^{m} u - Q^{p \overline{q}} T_{pm}^{t} u_{t \ov{q}} D^{m}u \\
& \ + Q^{p \overline{q}} u_{t} R_{p \ov q m}{}^{t} D^m u+Q^{p \overline{q}} u_m \omega^{m \ov \ell} D_{\ov \ell}\left( g_{p \overline{q}}-\chi_{p \overline{q}}\right) \\
&\ - Q^{p \overline{q}} u_m \omega^{m \ov{\ell}} \ov {T_{q \ell}^{t}}u_{p \ov t}
+|D D u|_{Q}^{2}+|D \ov{D} u|_{Q}^{2}\\
&\ - Q^{p \overline{q}} D_m (a_p u_{\ov q} +a_{\ov q} u_p)D^m u
- Q^{p \overline{q}} D_{\ov \ell} (a_p u_{\ov q} +a_{\ov q} u_p) u_m \omega^{m \ov \ell}.
\end{aligned}
\end{equation}
Using the Cauchy inequality, we have
\begin{align}\nonumber
&\ | Q^{p \overline{q}} T_{mp}^{t} u_{t \ov{q}} D^{m}u|
+|Q^{p \overline{q}} u_{m} \omega^{m \overline{\ell}} \overline{T_{q \ell}^{t}} u_{p \overline{t}}| \\ \nonumber
&\ +| Q^{p \overline{q}} D_m (a_p u_{\ov q} +a_{\ov q} u_p)D^m u|
+|Q^{p \overline{q}} D_{\ov \ell} (a_p u_{\ov q} +a_{\ov q} u_p) u_m \omega^{m \ov \ell}| \\ \nonumber
\leq &\ \frac{1}{2}|D \ov{D} u|_{Q }^{2}+\frac{1}{2}|DDu|_{Q }^{2}+C\mathcal{ Q }.
\end{align}
Substituting the above inequality into \eqref{(Du)_pq 2}, we get
\begin{equation}
\begin{aligned}
& Q^{p \overline{q}}|D u|_{p \overline{q}}^{2} \\
\geq &\ D_{m}(Q) u_{\ov{\ell}}\omega^{m \ov{\ell}}+D_{\ov{\ell}}(Q)u_m \omega^{m \ov{\ell}}
+\frac{1}{2}|DDu|_{ Q }^2+\frac{1}{2}|D\ov{D} u|_{Q }^2-C\mathcal{ Q }.
\end{aligned}
\end{equation}
Using the differential equation \eqref{differential equ}, we obtain
\begin{equation}
\begin{aligned}\label{derivative of Du'}
Q^{p \overline{q}}|D u|_{p \overline{q}}^{2}
\geq &\; 2 \operatorname{Re}\left\{\sum_{p, m}\left(D_{p} D_{m} u D_{\overline{p}} u+D_{p} u D_{\overline{p}} D_{m} u\right) \psi_{v_{m}}\right\}\\
&\; -C-C\mathcal{Q}+\frac{1}{2}|D D u|_{ Q }^{2}+\frac{1}{2}|D \ov{D} u|_{ Q }^{2}.
\end{aligned}
\end{equation}
We also compute that
\[
Q^{p \bar q} g_{p \bar q} = \sum_{s=1}^{k} \alpha_s \sigma_s^{p \bar q} g_{p\bar q}
= \sum_{s=1}^{k} \alpha_s s \sigma_s (\lambda) \leq k\psi.
\]
Hence,
\begin{equation}
\label{D-u}
\begin{aligned}
- Q^{p\bar q } u_{p\bar q} = &\ Q^{p\bar q} (\chi_{p\bar q} + a_p u_{\bar q} + a_{\bar q} u_p - g_{p\bar q})\\
\geq &\ \varepsilon \mathcal{Q} + Q^{p\bar q} ( a_p u_{\bar q} + a_{\bar q} u_p ) - k\psi.
\end{aligned}
\end{equation}

\section{Proof of Theorem \ref{thm1}}

We denote by $\lambda_{1}, \lambda_{2}, \dots, \lambda_{n}$ the eigenvalues of the Hermitian endomorphism
${g^i}_j = {\omega^{i} }_{\bar k } g_{j \overline{k}}$.
When $k=1$, it follows from the equation \eqref{eqn} that $\Delta_\omega u$ is bounded, and the desired estimate follows in turn from $g\in \Gamma_2 \subset \{\lambda| \sum \lambda_{i_1} + \cdots + \lambda_{i_{n-1}} \geq 0\}$. Henceforth, we assume that $k \geq 2$.

Since $g \in \Gamma_{k+1}$, by Lemma \ref{sigmak-1}, there exists a positive constant $K_0$ depending on $\psi$, $\sup_{M}|u|$, $\sup_{M}|Du|$ and $\alpha_k$ such that
 $\lambda_i (g) + K_0 \geq 0$ for any $1\leq i \leq n$.
We apply the maximum principle to the following test function:
\begin{equation}
\label{test}
G=\log P_{m}+\varphi(|Du|^{2})+\phi(u),
\end{equation}
where $P_{m}=\sum_{j} \kappa_{j}^{m}$ and $\kappa_j = \lambda_j (g) + K_0$. Here, $\varphi$ and $\phi$ are positive functions to be determined later,
which satisfy the following assumptions
\begin{equation}
\label{testfun}
\varphi''- 2\phi''(\frac{\varphi'}{\phi'})^{2} \geq 0, \;\varphi' > 0, \; \phi' < 0, \; \phi''>0.
\end{equation}
We may assume that the maximum of $G$ is achieved at some point $p_0 \in M$.
We choose the coordinate system centered at $p_0$ such that $\omega=\sqrt{-1} \sum \delta_{k \ell} dz^{k} \wedge d \overline{z}^{\ell}$,
$g_{i \ov{j}}$ is diagonal,
and the eigenvalues of the Hermitian endomorphism ${g^i}_j$ are ordered as
 $\lambda_{1} \geq \lambda_{2} \geq \cdots \geq \lambda_{n}$.

Differentiating $G$ ,
we first obtain the critical equation
\begin{equation}\label{critical equ}
\frac{D P_{m}}{P_{m}}+\varphi' D|D u|^{2}+\phi' Du=0.
\end{equation}
Differentiating $G$ a second time, using \eqref{symmetric func 2th deriv} and contracting with $Q^{p \overline{q}}$ yields
\begin{equation}\label{diffrrentiate G second time}
\begin{aligned}
0 \geq &\ \frac{m}{P_m}  \sum_j \kappa_j^{m-1} Q^{p \ov{p}} D_{\ov p} D_{p} g_{j \ov{j}}+\frac{m Q^{p \ov{p}}}{P_m} (m-1)\sum_j \kappa_j^{m-2} |D_p g_{j \ov{j}}|^2\\
&\ +\frac{m Q^{p \ov{p}}}{P_m} \sum_{i \neq j} \frac{\kappa_i^{m-1} - \kappa_j^{m-1}}{\kappa_i - \kappa_j} |D_p g_{i \ov{j}}|^2
+ Q^{p \ov{p}} \left( \phi'' D_p u D_{\ov p}u+\phi' u_{p \ov p} \right)\\
&\ + Q^{p \ov{p}}\left( \varphi'' D_p |Du|^2 D_{\ov p}|Du|^2+\varphi'|Du|_{p \ov p}^2 \right)- \frac{|D P_m|^2_Q}{P_m^2}.
\end{aligned}
\end{equation}
Here we used the notation introduced in Section 2.

Using the critical equation \eqref{critical equ}, we obtain
\begin{equation}
\begin{aligned}\label{1th order term}
D_p u D_{\ov p}u
\geq &\ \frac{1}{2 (\phi')^2 } \frac{|D_p P_m|^2}{P_m ^2 }-(\frac{\varphi^{\prime}}{\phi^{\prime}})^2 \big|D_p |Du|^2 \big|^2.
\end{aligned}
\end{equation}
Substituting  \eqref{1th order term} into \eqref{diffrrentiate G second time} and by \eqref{dif-eqn2}, \eqref{dif-eqn2-1},
\eqref{derivative of Du'}, \eqref{D-u}
\begin{equation}
\label{main inequality 1}
\begin{aligned}
0 \geq &\ -\frac{C\sum_{j}\kappa_j^{m-1}}{P_m} (1+|DDu|^2+|D\ov{D}u|^2+(1+\lambda_1)\mathcal{Q}+\lambda_1) \\
&\ +\frac{\sum_j\kappa_j^{m-1}}{P_m} \Big(-Q^{p\overline{q},r\overline{s}} D_{\ov{j}} g_{r \overline{s}}{D_{j}g_{p \overline{q}}}+\sum_{\ell}{\psi_{v_\ell}}g_{j \overline{j} \ell}
+\sum_{\ell}{\psi_{\overline{v}_{\ell}}}g_{j \overline{j}\overline{\ell}}\\
&\ - Q^{p\overline{p}}(T_{pj}^au_{a \ov p \ov j}+\ov{T_{pj}^a}u_{p \ov a j}) \Big)
+ \frac{m-1}{P_m}  \sum_{j}\kappa_j^{m-2} Q^{p\overline{p}}|{D_p}g_{j \overline{j}}|^2\\
&\ + \frac{ Q^{p\overline{p}} }{P_m} \sum_{i\ne j}{\frac{\kappa_i^{m-1}-\kappa_j^{m-1}}{\kappa_i-\kappa_j}}|{D_p}{g_{i \overline{j}}}|^2 - \Big(1-\frac{\phi''}{2({\phi'})^2}\Big) \frac{|DP_m|_{Q}^2}{mP_m^2} \\
&\  +\frac{\varphi'}{m}\bigg( \frac{1}{4}|DDu|_{Q}^2+\frac{1}{4}|D \ov{D} u|_{Q}^2 \bigg)
+\bigg(-\frac{\phi'}{m}\varepsilon-C\frac{\varphi'}{m} \bigg)\mathcal{Q}\\
&\ +2\frac{\varphi'}{m}\operatorname{Re}\Big( \sum_{p,m}( u_{mp} u_{\bar p}+u_pu_{m\bar p})\psi_{v_m} \Big)
- 2\frac{\phi'}{m} Q^{p\bar p} \operatorname{Re} \{a_p u_{\bar p} \}\\
&\ +k\frac{\phi'}{m}\psi-C\frac{\varphi'}{m}+\frac{Q^{p\overline{p}}}{m}\Big(\varphi''-\phi'' \frac{\varphi'^2}{\phi'^2} \Big) \bigg|D_p|Du|^2\bigg|^2\\
&\ + \frac{\sum_{j}\kappa_j^{m-1}}{P_m} Q^{p\bar p} (a_j u_{\ov j p \ov p}+a_{\ov j} u_{j p \ov p}
- a_p u_{\ov p j \ov j} - a_{\ov p} u_{pj \ov j})
\end{aligned}
\end{equation}
as long as $\frac{\varphi'}{4m} \geq \frac{C}{\kappa_1}$. Here we used $\varphi' > 0$ and $\phi' < 0$ in \eqref{testfun}.

From the critical equation \eqref{critical equ}, we obtain
\begin{equation}
\begin{aligned}
&\frac{1}{P_m}\sum_j \kappa_j ^{m-1} \sum_{\ell} (\psi_{v_{\ell}}D_{\ell} g_{j \ov j}+\psi_{\ov{v}_{\ell}}D_{\ov{\ell}} g_{j \ov j})\\
&+2\frac{\varphi'}{m}\operatorname{Re}\Big( \sum_{p,m}( u_{mp} u_{\bar p}+u_pu_{m\bar p})\psi_{v_m} \Big) \\
=&\ -\frac{\phi'}{m} \sum_{\ell}(u_{\ell} \psi_{v_{\ell}}+
u_{\ov{\ell}} \psi_{\ov{v}_{\ell}})-\frac{2 \varphi'}{m}\operatorname{Re}(\sum_{p,k,m}\psi_{v_m}T_{pm}^{k}u_k u_{\ov p}) \\
\geq &\ C\frac{\phi'}{m}-C\frac{\varphi'}{m},
\end{aligned}
\end{equation}
and
\begin{equation}
\begin{aligned}
2\frac{\phi'}{m} Q^{p\bar p} \operatorname{Re} \{a_p u_{\bar p} \}
= &\ \Big|\frac{2 Q^{p \overline{p}}}{m} \operatorname{Re} \left\{ a_p (\frac{D_{\ov p} {{P}_m}}{{P}_m}+\varphi' D_{\ov p}|Du|^2) \right\} \Big| \\
\leq &\ \frac{\phi''}{4{\phi'}^2} \frac{|D{P_m}|_{Q}^2}{m{P}^2_m}+ \frac{ Q^{p \overline{p}}}{m} {\varphi'}^2 \frac{\phi''}{\phi'^2}\bigg|D_p|Du|^2\bigg|^2+C\frac{\phi'^2}{\phi''} \mathcal{ Q}.
\end{aligned}
\end{equation}

With the above calculations and \eqref{identity}, we have
\begin{equation}
\begin{aligned}\label{main inequality 2}
0 \geq
&\ \frac{1}{{P_m}}\sum_j{\kappa}_j^{m-1}\bigg(- Q^{p\overline{p},q\overline{q}} {D_{j} g_{p \overline{p}}} D_{\ov{j}} g_{q \overline{q}}+ Q^{p\overline{p},q\overline{q}}|D_j g_{q \ov p}|^2 \bigg) \\
&\ -\frac{2 Q^{p\overline{p}}}{P_m}\sum_j \kappa_j^{m-1}
\operatorname{Re}\Big(\ov{T_{pj}^s}u_{p \ov s j}\Big)
+ \frac{(m-1) Q^{p\overline{p}}}{P_m} \sum_j \kappa_j^{m-2}|{D_p} g_{j \overline{j}}|^2 \\
&\ +\frac{ Q^{p\overline{p}}}{{{P}_m}} \sum_{i\ne j}{\frac{\kappa_i^{m-1}
- \kappa_j^{m-1}}{\kappa_i-\kappa_j}}|{D_p}{g_{i \overline{j}}}|^2
-\bigg(1-\frac{\phi''}{ 4 (\phi')^2}\bigg)\frac{ |DP_m|_{Q}^2 }{m P_m^2} \\
&\ + \frac{\varphi'}{4 m} \bigg(|DDu|_{ Q }^2+|D \ov{D} u|_{ Q }^2 \bigg)
+C\frac{\phi'}{m}-C\frac{\varphi'}{m}-C\\
&\ +\bigg(-\frac{\phi'}{m}\varepsilon-C\frac{\varphi'}{m}-C-C\frac{\phi'^2}{\phi''}\bigg)\mathcal{ Q }
-\frac{C}{\lambda_1}(|DDu|^2+|D\ov{D}u|^2)\\
&\  + \frac{ Q^{p\overline{p}}}{{P_m}}\sum_j \kappa_j^{m-1}
\Big(a_j u_{\ov j p \ov p}+a_{\ov j} u_{j p \ov p}-a_p u_{\ov p j \ov j}-a_{\ov p} u_{pj \ov j} \Big) ,
\end{aligned}
\end{equation}
where we used $\varphi''- 2 \phi'' \frac{\varphi'^2}{\phi'^2} \geq 0$ in \eqref{testfun} and $\kappa_{1} \gg 1$.

Let
\begin{equation}
\begin{aligned}
{\tilde A_j} = &\; \frac{1}{P_m}\kappa_j^{m-1} \sum_{p,q} Q^{p \overline{p}, q \overline{q}} D_{j} g_{p \overline{p}} D_{\overline{j}} g_{q \overline{q}},\;\;\;\;
{\tilde B_q}= \frac{1}{P_{m}} \sum_{j,p} \kappa_{j}^{m-1} Q^{p \overline{p}, q \overline{q}}\left|D_{j} g_{q \overline{p}}\right|^{2},
\nonumber\\
C_p=&\; \frac{m-1}{P_m} Q^{p \overline{p}} \sum_{j} \kappa_{j}^{m-2} \left|D_{p} g_{j \overline{j}}\right|^{2},\;\;\;\;\;\;\;
{\tilde D_p}=\frac{ Q^{p \overline{p}}}{P_m} \sum_{j \neq i} \frac{\kappa_{i}^{m-1}-\kappa_{j}^{m-1}}{\kappa_{i}-\kappa_{j}}\left|D_{p} g_{i \overline{j}}\right|^{2},\nonumber\\
E_i=&\;\frac{m}{P^2_m} Q^{i \ov i}|\sum_p \kappa_p^{m-1}D_i g_{p \ov p}|^2,\;\;\;\;\;\;\;\;\;\;\;\;\;\;
H_p= \frac{2 Q^{p \overline{p}}}{P_{m}} \sum_{j,s} \kappa_{j}^{m-1}\operatorname{Re}(\ov{T_{pj}^s}u_{p \ov s j})\nonumber,
\end{aligned}
\end{equation}
and
\[
I_p= \frac{Q^{p\overline{p}}}{P_m}\sum_j \kappa_j^{m-1}\Big(a_j u_{\ov j p \ov p}+a_{\ov j} u_{j p \ov p}-a_p u_{\ov p j \ov j}-a_{\ov p} u_{pj \ov j}\Big).
\]
Then \eqref{main inequality 2} becomes
\begin{equation}
\begin{aligned}\label{main inequalit 3}
0\geq &\ -\sum_j {\tilde A_j}+\sum_q {\tilde B_q}+\sum_p C_p+\sum_p {\tilde D_p}-\sum_p H_p + \sum_p I_p\\
&\ -\Big(1-\frac{\phi^{\prime \prime}}{ 4 (\phi^{\prime})^{2}}\Big)\sum_i E_i
+\frac{\varphi^{\prime}}{4 m}\left(|D D u|_{Q }^{2} + |D \ov{D} u|_{Q }^{2}\right) \\
&\ +\left(-\frac{\phi^{\prime}}{m} \varepsilon-C \frac{\varphi^{\prime}}{m}-C-C\frac{\phi'^2}{\phi''}\right) \mathcal{Q}
-C\Big(- \frac{\phi^{\prime}}{m}+\frac{\varphi^{\prime}}{m}+1\Big)\\
&\ -\frac{C}{\lambda_{1}}\left(|D D u|^{2}+|D \ov{D} u|^2\right).
\end{aligned}
\end{equation}

We first deal with the torsion term $H_p$.
For any $0< \tau <1$, by \eqref{order} we can estimate
\begin{equation}\nonumber
\begin{aligned}
H_p
\leq &\; \frac{2 Q^{p\ov p}}{{{P}_m}} \sum_{j,s} \kappa_j^{m-1} |\ov{T_{pj}^s} D_p g_{j\ov s }| + C Q^{p\ov p} +
\frac{C}{\kappa_1} Q^{p\ov p}\sum_{j} |u_{jp}|^2\\
\leq &\;  \frac{\tau}{2} \frac{Q^{p\ov p}}{{P}_m} \sum_{j,s} \kappa_j^{m-2} |D_p g_{j\ov s }|^2+ \frac{C}{\tau} Q^{p\ov p}+\frac{C}{\kappa_1} Q^{p\ov p}\sum_{j} |u_{jp}|^2.
\end{aligned}
\end{equation}
Similarly, we can estimate
\begin{equation}\nonumber
\begin{aligned}
|I_p|
\leq &\; \frac{\tau}{4} \frac{Q^{p\ov p}}{{{P}_m}} \sum_{j}  {\kappa}_j^{m-2}
\Big(|D_p g_{j\ov p }|^2 + |D_p  g_{j\ov j }|^2\Big)
+ \frac{C}{\tau} Q^{p\ov p}+\frac{C}{\kappa_1} Q^{p\ov p} \sum_{j} |u_{jp}|^2\\
\leq &\; \frac{\tau}{2} \frac{Q^{p\ov p}}{{{P}_m}} \sum_{j,s}  \kappa_j^{m-2} |D_p g_{j\ov s }|^2
+ \frac{C}{\tau} Q^{p\ov p}+\frac{C}{\kappa_1} Q^{p\ov p} \sum_{j} |u_{jp}|^2.
\end{aligned}
\end{equation}
By direct computation, we have
\begin{align}
 \tilde D_p = \frac{Q^{p\ov p}}{P_m} \sum_{j\neq i} \sum_{s=0}^{m-2}\kappa_i^{m-2-s}\kappa_j^s |D_p g_{i\ov j}|^2
 \geq \frac{Q^{p\ov p}}{P_m} \sum_{j\neq i}  \kappa_i^{m-2}|D_p g_{i \ov j}|^2.
\end{align}
Now we have
\begin{equation}
\label{HICD}
\begin{aligned}
&\; - \sum_p H_p +\sum_p I_p+ \sum_p C_p + \sum_p \tilde D_p \\
\geq &\; (1-\tau) \sum_p \tilde D_p + (1 - \tau)\sum_p C_p - \frac{C}{\tau}\mathcal{ Q }-\frac{C}{\kappa_1}|DDu|_{Q }^2,
\end{aligned}
\end{equation}
as $m>2$ will be chosen large enough.
Substituting \eqref{HICD} into \eqref{main inequalit 3} yields
\begin{equation}
\begin{aligned}
\label{main inequality 4}
0\geq &\; -\sum_j \tilde{A_j}+\sum_q {\tilde B_q}
+(1 - \tau) \bigg(\sum_p C_p + \sum_p {\tilde D_p} \bigg)\\
&\;-\left(1-\frac{\phi^{\prime \prime}}{ 4 (\phi^{\prime})^{2}}\right)\sum_i E_i
+\frac{\varphi^{\prime}}{4m}\left(|D D u|_{Q}^{2}+|D \ov{D} u|_{Q}^{2}\right) \\
&\ +\left(-\frac{\phi^{\prime}}{m} \varepsilon -C \frac{\varphi^{\prime}}{m}-\frac{C}{\tau} - C\frac{\phi'^2}{\phi''}\right) \mathcal{Q}
-C(\varphi'- \phi'+1)\\
&\ -\frac{C}{\lambda_{1}}\left(|D D u|^{2}+|D \ov{D} u|^2\right).
\end{aligned}
\end{equation}

Now we shall denote
\[\begin{aligned}
A_j = &\; \frac{1}{P_m} \kappa_j^{m-1} \Big(K |D_j Q|^2 - Q^{p\ov p, q\ov q} D_j g_{p\ov p}D_{\ov j} g_{q\ov q}\Big),\\
B_q = &\; \frac{1}{P_m} \sum_p \kappa_p^{m-1} Q^{p\ov p, q\ov q}|D_q g_{p\ov p}|^2,\\
D_i = &\; \frac{1}{P_m} \sum_{p\neq i} Q^{p\ov p} \frac{\kappa_p^{m-1} - \kappa_i^{m-1}}{\kappa_p -\kappa_i} |D_i g_{p\ov p}|^2.
\end{aligned}\]
Taking $\ell=1$ and $\theta = 1/2$ in Lemma \ref{GRW}, we see
that $A_j \geq 0$ for $1 \leq j \leq n$ for $K>(1 + \beta )(\inf \psi)^{-1}$ if $2 \leq k\leq n$.
Define
\[
H_{j \ov q p} = D_j(\chi_{p\ov q}+a_p u_{\ov q}+a_{\ov q} u_p) - D_p (\chi_{j\ov q}+a_j u_{\ov q}+a_{\ov q} u_j).
\]
Note that $|H_{j\ov q p}|\leq C+C\lambda_1$, since $u_{pj}-u_{jp}=T_{pj}^{k} u_{k} \leq C$.
For any constant $0<\tau <1$, we can estimate
\begin{equation}\nonumber
\begin{aligned}
\sum_q \tilde B_q
\geq &\;
\frac{1}{P_m} \sum_{j,q} \kappa_j^{m-1} Q^{j\ov j, q\ov q} |D_j g_{q \ov j }|^2\\
= &\; \frac{1}{P_m} \sum_{j,q} \kappa_j^{m-1} Q^{j\ov j, q\ov q} |D_q g_{j\ov j} - T_{jq}^a u_{a \ov j} + H_{j\ov j q}|^2\\
\geq &\; \frac{1}{P_m}\sum_{q,j} \kappa_j^{m-1} Q^{j\ov j, q\ov q}
\Big( (1-\tau)|D_q g_{j \ov j}|^2 - \frac{1}{\tau} | H_{j\ov j q}- T_{jq}^a u_{a\ov j} |^2\Big)\\
= &\; (1-\tau) \sum_q B_q - \frac{1}{\tau P_m} \sum_{q,j} \kappa_j^{m-1} Q^{j\ov j, q\ov q}
| H_{j\ov j q} - T_{jq}^a u_{a\ov j} |^2.
\end{aligned}
\end{equation}
By Lemma \ref{sigmak} (4) and our assumption $\lambda \in \Gamma_{k+1}$,
we have $\sigma_{s-1} (\lambda | jq) > 0$ for any $1\leq j, q \leq n$ and $1\leq s \leq k$.
We then see $\sigma_s^{j\bar j,q \bar q} \kappa_j
\leq \sigma_{s}^{q\bar q} + K_0 \sigma_s^{j\bar j,q \bar q}$,
which implies that
\[
\sigma_s^{j\bar j,q \bar q} \kappa_j^{m-1} \leq \sigma_{s}^{q\bar q} \sum_{l=0}^{m-2} \kappa_j^{l}K_0^{m-2-l}
+ K_0^{m-1} \sigma_s^{j\bar j,q \bar q}.
\]
By Lemma \ref{sigmak} (2) and (6), we have
\[
\sum_{j,q} \sigma_s^{j\bar j,q \bar q} = c_{s,n} \sigma_{s-2} (\lambda) \leq C \kappa_1^{s-2},
\]
where $c_{s,n} = (n-s+1 )(n-s+2)$.
Now, we obtain
when $m\geq k$ and $\kappa_1$ is large enough,
\[
\frac{1}{\tau P_m} \sum_{q,j} \kappa_j^{m-1} Q^{j\ov j, q\ov q} |H_{j\ov j q}- T_{jq}^a u_{a\ov j} |^2
\leq \frac{C \kappa_1^2}{\tau P_m} \sum_{q,j} \kappa_j^{m-1} Q^{j\ov j, q\ov q}
\leq \frac{C}{\tau}\mathcal{Q} + \frac{C}{\tau},
\]
We finally arrive at
\begin{align}\nonumber
\sum_q \tilde B_q
\geq &\; (1-\tau)\sum_q B_q -\frac{C}{\tau}\mathcal{Q} - \frac{C}{\tau}\nonumber.
\end{align}
Similarly, we can estimate
\[\begin{aligned}
\sum_p \tilde D_p \geq &\; \frac{1}{P_m} \sum_{j\neq i} Q^{j\ov j}
\frac{\kappa_i^{m-1}- \kappa_j^{m-1}}{\kappa_i -\kappa_j}|D_j g_{i\ov j}|^2\\
\geq &\; \frac{1}{P_m} \sum_{j\neq i} Q^{j\ov j} \frac{\kappa_i^{m-1}- \kappa_j^{m-1}}{\kappa_i -\kappa_j}
\Big((1-\tau) |D_i g_{j\ov j}|^2 - \frac{C}{\tau} \kappa_1^2 \Big)\\
\geq &\; (1-\tau)\sum_i D_i - \frac{C}{\tau} \mathcal{Q}.
\end{aligned}\]
Note that $\frac{\kappa_j^{m-1}}{P_m} |D_j Q|^2 \leq \frac{C}{\kappa_1} (|DDu|^2+ |D\ov D u|^2)$.
Then \eqref{main inequality 4} becomes
\begin{equation}
\begin{aligned}
\label{main inequality 6}
0\geq &\; (1-\tau)^2 \sum_i \Big(  A_i+  B_i+ C_i+  D_i\Big)-\left(1-\frac{\phi^{\prime \prime}}{4 (\phi')^{2}}\right)\sum_i E_i\\
&\; +\frac{\varphi^{\prime}}{4m}\left(|D D u|_{Q}^{2}+ |D \ov{D} u|_{Q}^{2}\right)
 -\frac{C(K)}{\kappa_{1}}\left(|D D u|^{2}+|D \ov{D} u|^2\right)\\
&\ +\left(-\frac{\phi^{\prime}}{m} \varepsilon-C \frac{\varphi^{\prime}}{m} - \frac{C}{\tau}
- C\frac{\phi'^2}{\phi''} \right) \mathcal{Q}
-C(\varphi'- \phi'+\frac{1}{\tau}),
\end{aligned}
\end{equation}
when $\lambda_1$ is sufficiently large.

Now we choose $\phi$ and $\varphi$ to satisfy \eqref{testfun}. Let $\varphi (t) = e^{Nt}$ and $\phi(s) = e^{M(-s+L)}$ where $L \geq |u|_{C^1}+1$ is a constant.
Then, we see
\[
\varphi''- 2 \phi'' \frac{\varphi'^2}{\phi'^2} = N^2 e^{Nt} - 2 \frac{N^2 e^{2Nt}}{e^{M(-s +L)}} >0, \quad \varphi' >0, \quad \phi' < 0, \quad \phi''>0,
\]
when $M \gg N > 1$, which shows the assumption \eqref{testfun} is satisfied.
Choosing
\[
2 \tau = \frac{\phi''}{4(\phi')^2} = \frac{1}{4 e^{M(-u(p) + L)}},
\]
we obtain that $(1-\tau)^2 \geq 1-\frac{\phi^{\prime \prime}}{ 4 \phi'^{2}} $.
By Lemma \ref{sigmak} (3) and (5), we have
\[
Q^{i \overline{i}} \geq Q^{1 \overline{1}} \geq \frac{Q}{n \lambda_{1}} \geq \frac{1}{C \lambda_{1}}
 \]
for any fixed $i$, where $C$ depends on $\inf \psi > 0$ and other known data. We can estimate
\[
|D D u|_{Q }^{2} + |D \ov{D} u|_{Q }^{2} \geq \frac{1}{C \lambda_{1}}\left(|D D u|^{2}+|D \overline{D} u|^{2}\right) \geq \frac{1}{C\lambda_1} |DDu|^2 + \frac{\lambda_1}{C}.
\]
Now \eqref{main inequality 6} becomes
\begin{equation}
\label{main inequality 7}
\begin{aligned}
0\geq &\; (1-\tau)^2 \sum_i (A_i + B_i + C_i + D_i - E_i)\\
&\; + \Big(\frac{\varphi'}{m C} - C(K)\Big)\lambda_1 + \frac{1}{\lambda_1} \Big(\frac{\varphi'}{m C} - C(K)\Big) |DDu|^2 \\
&\; +\left(-\frac{\phi^{\prime}}{m} \varepsilon-C \frac{\varphi^{\prime}}{m} - \frac{C}{\tau} - C\frac{\phi'^2}{\phi''} \right) \mathcal{Q}
-C(\varphi'- \phi'+ \frac{1}{\tau}).
\end{aligned}
\end{equation}
Taking $N$ large enough, we can ensure that $\frac{\varphi'}{m C} - C(K)>0$.
For fixed $N$, it follows that
\[-\frac{\phi^{\prime}}{m} \varepsilon-C \frac{\varphi^{\prime}}{m} - \frac{C}{\tau} - C\frac{\phi'^2}{\phi''}
=\frac{M}{m} \varepsilon \phi - C \frac{N}{m} \varphi - C \phi > 0 \]
when $M\gg N$.

\textbf{Claim}: For sufficiently large $m$, we may assume
\[
A_{i}+B_{i}+C_{i}+D_{i}-E_{i} \geq 0, \quad \forall i=1, \ldots, n.
\]
This leads to
\[\begin{aligned}
0\geq &\; \Big(\frac{\varphi'}{m C} - C(K)\Big)\lambda_1  - C (\varphi'- \phi'+ \phi),
\end{aligned}\]
which finally implies an upper bound of $\lambda_1$.

We now prove the claim to finish the proof of Theorem \ref{thm1}.

\begin{lemma}
\label{lemm-1}
For sufficiently large $m$, the following estimates hold:
\begin{equation}
\label{lemm-11}
P_m^2 (B_1+ C_1 + D_1 - E_1 ) \geq P_m \kappa_{1}^{m-2} \sum_{j\neq 1} Q^{j\bar j} | D_1 g_{j\bar j}|^2
- Q^{1\bar 1} \kappa_1^{2m-2} |D_1g_{1\bar 1}|^2,
\end{equation}
and for any index $i\neq 1$,
\begin{equation}
\label{lemm-12}
P_m^2 (B_i+ C_i + D_i - E_i ) \geq 0.
\end{equation}
\end{lemma}

\begin{proof}
For any index $i$, we compute that
\[
P_m (B_i + D_i) =  \sum_{j\neq i}  Q^{j\bar j,i\bar i} \kappa_j^{m-1} |D_i g_{j\bar j}|^2+
  \sum_{j \neq i } Q^{j\bar j} \sum_{l=0}^{m-2} \kappa_i^{m-2-l}\kappa_j^l |D_i g_{j\bar j}|^2.
\]
Note that $\kappa_j  Q^{j\bar j,i\bar i} + Q^{j\bar j} \geq  Q^{i\bar i}$.
To see this, we compute that
\[\begin{aligned}
\kappa_j  \sigma_s^{j\bar j,i\bar i} + \sigma_s^{j\bar j}
& = \lambda_j \sigma_{s-2} (\lambda | ij) + K_0 \sigma_{s-2} (\lambda | ij) + \sigma_{s-1} (\lambda | j)\\
& = \sigma_{s-1} (\lambda | i) - \sigma_{s-1} (\lambda | ij) + K_0 \sigma_{s-2} (\lambda | ij) + \sigma_{s-1} (\lambda | j)\\
& = \sigma_{s-1} (\lambda | i) + ( \lambda_i + K_0 ) \sigma_{s-2} (\lambda | ij) \\
& \geq \sigma_{s-1} (\lambda | i),
\end{aligned}\]
for every $1\leq s\leq k$.
We therefore have
\[
P_m (B_i + D_i)\geq  Q^{i\bar i} \sum_{j\neq i}  \kappa_j^{m-2 } | D_i g_{j\bar j}|^2
+ \sum_{j \neq i } Q^{j\bar j} \sum_{l=0}^{m-3} \kappa_i^{m-2-l}\kappa_j^l | D_i g_{j\bar j}|^2.
\]
It follows that
\[\begin{aligned}
P_m(B_i + C_i + D_i) \geq &\; m Q^{i\bar i}\sum_{j\neq i} \kappa_j^{m-2} | D_i g_{j\bar j} |^2
+ (m-1) Q^{i\bar i}\kappa_i^{m-2} | D_i g_{i\bar i}|^2\\
&\; + \sum_{j \neq i } Q^{j\bar j} \sum_{l=0}^{m-3} \kappa_i^{m-2-l} \kappa_j^l | D_i g_{j\bar j}| ^2
\end{aligned}\]
For $E_i$, we have
\[\begin{aligned}
P_m^2 E_i = &\; m Q^{i\bar i} \sum_{j\neq i} \kappa_j^{2m-2} | D_i g_{j\bar j}|^2
+ m Q^{i\bar i} \kappa_i^{2m-2} | D_i g_{i\bar i}|^2\\
&\; + m Q^{i\bar i} \sum_{p\neq i}\sum_{q\neq p, i} \kappa_p^{m-1} \kappa_q^{m-1} D_i g_{p\bar p} \ol{ D_i g_{q\bar q} }\\
&\; + 2 m Q^{i\bar i} \operatorname{Re}  \sum_{j\neq  i} \kappa_i^{m-1} D_i g_{i\bar i} \kappa_j^{m-1}  \ol{ D_i g_{j\bar j} }.
\end{aligned}\]
By Cauchy-Schwarz inequality, we know
\[
\begin{aligned}
& \sum_{p\neq i}\sum_{q\neq p, i} \kappa_p^{m-1} \kappa_q^{m-1} D_i g_{p\bar p} \ov{ D_i g_{q\bar q} }\\
\leq &\; \sum_{p\neq i}\sum_{q\neq p, i} \frac{1}{2}(\kappa_p^{m-2} \kappa_q^m | D_i g_{p\bar p}| ^2
+ \kappa_q^{m-2} \kappa_p^m |D_i g_{q\bar q}|^2 ).
\end{aligned}
\]
By the symmetry of $p$ and $q$ in the above inequality, we obtain
\[
\sum_{p\neq i}\sum_{q\neq p, i} \kappa_p^{m-1} \kappa_q^{m-1} D_i g_{p\bar p}  \ov{D_i g_{q\bar q} } \leq
\sum_{p\neq i}\sum_{q\neq p, i} \kappa_p^{m-2} \kappa_q^m | D_i g_{p\bar p} |^2.
\]
Therefore,
\begin{equation}
\label{BCDE}
\begin{aligned}
& P_m^2(B_i + C_i + D_i - E_i)\\
\geq &\; Q^{i\bar i}\sum_{j\neq i} [m P_m - m\kappa_j^m ]\kappa_j^{m-2} | D_i g_{j\bar j}|^2
+ P_m \sum_{j \neq i } Q^{j\bar j} \sum_{l=0}^{m-3} \kappa_i^{m-2-l} \kappa_j^l | D_i g_{j\bar j}|^2\\
&\; + [(m-1) P_m - m \kappa_i^m ] Q^{i\bar i}\kappa_i^{m-2} | D_i g_{i\bar i}|^2
- m Q^{i\bar i}\sum_{j\neq i}\sum_{q\neq j, i} \kappa_j^{m-2} \kappa_q^m | D_i g_{j\bar j} |^2 \\
&\; - 2 m Q^{i\bar i} \operatorname{Re }  \sum_{j\neq  i} \kappa_i^{m-1} D_i g_{i\bar i} \kappa_j^{m-1}   \ol{ D_i g_{j\bar j} }.
\end{aligned}
\end{equation}
By calculating, we see
$m P_m - m \kappa_j^m - m \sum_{q\neq j, i}  \kappa_q^m =  m \kappa_i^{m}$,
and we arrive at
\begin{equation}
\label{BCDE-1}
\begin{aligned}
P_m^2&\;(B_i + C_i + D_i - E_i)\\
\geq &\; m Q^{i\bar i}\sum_{j\neq i}  \kappa_i^m \kappa_j^{m-2} | D_i g_{j\bar j}|^2
+ [(m-1) P_m - m \kappa_i^m ] Q^{i\bar i} \kappa_i^{m-2} | D_i g_{i\bar i}|^2\\
 + &\; P_m \sum_{j \neq i } Q^{j\bar j} \sum_{l=0}^{m-3} \kappa_i^{m-2-l} \kappa_j^l | D_i g_{j\bar j} |^2
 - 2 m Q^{i\bar i} \operatorname{Re}  \sum_{j\neq  i}  \kappa_i^{m-1} D_i g_{i\bar i} \kappa_j^{m-1}  \ol{ D_i g_{j\bar j} }.
\end{aligned}
\end{equation}
We now estimate the third term in the right hand side of the above inequality.

Case A: $\lambda_i \geq \lambda_j$. Then $\kappa_i\geq \kappa_j$ and $Q^{j\bar j}\geq Q^{i\bar i}$. Hence
\[
 P_m  Q^{j\bar j} \sum_{l=1}^{m-3} \kappa_i^{m-2-l}\kappa_j^l \geq \kappa_i^m  Q^{i\bar i} \sum_{l=1}^{m-3} \kappa_i^{m-2-l}\kappa_j^l
 \geq (m-3) Q^{i\bar i} \kappa_i^m  \kappa_j^{m-2}.
\]

Case B: $\lambda_i \leq \lambda_j$. We further divide into two subcases.
If $\lambda_i \geq  K_0$,
\[\begin{aligned}
 P_m  Q^{j\bar j} \sum_{l=1}^{m-3} \kappa_i^{m-2-l}\kappa_j^l
 \geq &\;  \kappa_1^m \sum_{s=1}^{k} \alpha_s (\lambda_i \sigma_s^{i\bar i,j\bar j} + \sigma_{s-1}(\lambda| ij) ) \sum_{l=1}^{m-3} \kappa_i^{m-2-l}\kappa_j^{l}\\
 \geq &\;  \frac{1}{2}\kappa_1^m  \sum_{s=1}^{k} \alpha_s ( \kappa_i \sigma_s^{i\bar i,j\bar j} + 2 \sigma_{s-1}(\lambda| ij) ) \sum_{l=1}^{m-3} \kappa_i^{m-2-l}\kappa_j^{l}\\
 \geq &\;  \frac{1}{2}\kappa_1^m  \sum_{s=1}^{k} \alpha_s ( \kappa_j \sigma_s^{i\bar i,j\bar j} + 2 \sigma_{s-1}(\lambda| ij) ) \sum_{l=1}^{m-3} \kappa_i^{m-1-l}\kappa_j^{l-1}\\
 \geq &\;  \frac{1}{2}\kappa_1^m  \sum_{s=1}^{k} \alpha_s (\lambda_j \sigma_s^{i\bar i,j\bar j} + \sigma_{s-1}(\lambda| ij) ) \sum_{l=1}^{m-3} \kappa_i^{m-1-l}\kappa_j^{l-1}\\
 \geq &\; \frac{1}{2}(m-3) Q^{i\bar i} \kappa_i^m  \kappa_j^{m-2},
\end{aligned}\]
where we used $\sigma_{s-1} (\lambda |ij) > 0$ for $1 \leq s \leq k$ by our assumption $\lambda \in \Gamma_{k+1}$ in the third inequality.
If $\lambda_i \leq K_0$, for $k\leq l \leq [\frac{m-3}{2}]$, since $\lambda_1 \sigma_{s-1} (\lambda | 1) \geq \frac{s}{n} \sigma_s (\lambda)$,
we know
\[
\kappa_1^{l+1} Q^{j\bar j} = \kappa_1^{l+1} \sum_{s=1}^k \alpha_s \sigma_s^{j\bar j}
\geq  \frac{\kappa_1^l}{n} \sum_{s=1}^k \alpha_s  \sigma_s \geq \frac{\inf \psi }{n} \kappa_1^l,
\]
from which we obtain that
$\kappa_1^{l+1} Q^{j\bar j} \geq Q^{i\bar i}$ when $\kappa_1$ is sufficiently large
since $Q^{i\bar i} \leq C \lambda_1^{k-1}$.
We then obtain
\[\begin{aligned}
 P_m  Q^{j\bar j} \sum_{l=1}^{m-3} \kappa_i^{m-2-l}\kappa_j^l
\geq \sum_{l=k}^{m-3}  Q^{i\bar i} \kappa_1^{m-l-1}  \kappa_i^{m-2-l} \kappa_j^{l}
\geq  \frac{m-3}{2} Q^{i\bar i}  \kappa_i^m  \kappa_j^{m-2}
\end{aligned}\]
where in the last inequality we used $\kappa_1 \kappa_i^{m-2-l} \geq \kappa_i^{m}$ when $\kappa_1$ is sufficiently large.

Combining both cases, we have
\[
P_m Q^{j\bar j} \sum_{l=0}^{m-3} \kappa_i^{m-2-l}\kappa_j^l | D_i g_{j\bar j} |^2
\geq  \Big( \frac{m-3}{2}  Q^{i\bar i} \kappa_i^m  \kappa_j^{m-2}
+  P_m  Q^{j\bar j} \kappa_i^{m-2} \Big) | D_i g_{j\bar j} |^2.
\]
Using Cauchy-Schwarz inequality again, we see
\[
2 m \kappa_i^{m-1} D_i g_{i\bar i}  \kappa_j^{m-1}  \ol{ D_i g_{j\bar j} }
\leq \frac{3m-3}{2} \kappa_i^m \kappa_j^{m-2} | D_i g_{j\bar j} |^2 + (m-2) \kappa_i^{m-2} \kappa_j^{m} | D_i g_{i\bar i} |^2,
\]
where we also used $m^2 \leq \frac{3m-3 }{2}(m-2)$ when $m$ is sufficiently large.
Substituting the above two inequalities into \eqref{BCDE-1},
we arrive at
\begin{equation}
\label{BCDE-3}
\begin{aligned}
P_m^2&\; (B_i + C_i + D_i - E_i)\\
\geq &\; P_m  \kappa_i^{m-2}\sum_{j\neq i} Q^{j\bar j}  | D_i g_{j\bar j} |^2
- (m-2) Q^{i\bar i} \sum_{j\neq i} \kappa_i^{m-2} \kappa_j^{m} | D_i g_{i\bar i} |^2\\
& + [(m-1) P_m - m \kappa_i^m ] Q^{i\bar i}\kappa_i^{m-2} | D_i g_{i\bar i} |^2\\
\geq &\; [(m-1) P_m - m \kappa_i^m - (m-2) (P_m-\kappa_i^m)] Q^{i\bar i}\kappa_i^{m-2} | D_i g_{i\bar i} |^2 \geq 0,
\end{aligned}
\end{equation}
where we used $i\neq 1$ in the last inequality.
For the case $i=1$, by the first inequality of \eqref{BCDE-3}, we have
\[
\begin{aligned}
P_m^2&\;(B_1 + C_1 + D_1 - E_1)\\
\geq &\; P_m  \kappa_1^{m-2}\sum_{j\neq 1} Q^{j\bar j}  | D_1 g_{j\bar j} |^2
- \kappa_1^m  Q^{1\bar 1}\kappa_1^{m-2} | D_1 g_{1\bar 1} |^2.\\
\end{aligned}
\]
\end{proof}

\begin{lemma}
\label{lemm-2}
Suppose there exists $0<\delta \leq 1$ such that $\lambda_\mu\geq \delta \lambda_1$ for some $1 \leq \mu \leq k-1$.
There exists a sufficiently small positive constant $\delta'$ such that if $\lambda_{\mu+1} \leq \delta' \lambda_1$,
then
\[
A_1 + B_1 + C_1 + D_1 - E_1 \geq 0.
\]
\end{lemma}

\begin{remark}
With the above hypothesis,
if for some $1 \leq s \leq \mu$, $\alpha_s \neq 0$ in $Q$, we have from the equation that
\[
\psi \geq \alpha_s \sigma_s \geq \alpha_s \lambda_1 \cdots \lambda_s \geq \alpha_s \delta^{s -1} \lambda_1^s
\]
from which we can obtain a upper bound for $\lambda_1$. Henceforth, we may assume in the following proof that
$\alpha_1 = \alpha_2 = \cdots = \alpha_\mu = 0$ and $Q = \sum_{s= \mu+1}^k \alpha_s \sigma_s$.
\end{remark}

\begin{proof}
By our assumption, we can always assume that $\lambda_\mu > 1$, otherwise we can obtain an upper bound for $\lambda_1$.
By Lemma \ref{lemm-1},  we see
\begin{equation}
\label{ABCDE'}\begin{aligned}
P_m^2 &\; (A_1 + B_1 + C_1 + D_1 - E_1) \\
\geq &\; P_m^2 A_1 + P_m  \kappa_1^{m-2}\sum_{j\neq 1} Q^{j\bar j}  | D_1 g_{j\bar j} |^2
- \kappa_1^m  Q^{1\bar 1}\kappa_1^{m-2} | D_1 g_{1\bar 1} |^2.
\end{aligned}\end{equation}
Choosing $\theta = \frac{1}{2}$ in Lemma \ref{GRW}, we have for $\mu<k$
\begin{equation}
\label{ABCDE-1}
\begin{aligned}
P_m^2 A_1
\geq &\; \frac{P_m \kappa_1^{m-1} Q}{S_\mu^2}
\Big(\big(1+\frac{\beta}{2}\big) |\sum_j S_\mu^{j\bar j} D_1 g_{j\bar j} |^2
- S_\mu S_\mu^{p\bar p,q \bar q} D_1 g_{p\bar p} D_{\bar 1} g_{q\bar q} \Big)\\
\geq &\; \frac{P_m \kappa_1^{m-1} Q}{S_\mu^2}  \Big(\big(1+\frac{\beta}{2}\big) \sum_j|S_\mu^{j\bar j} D_1 g_{j\bar j}|^2
+\frac{\beta}{2} \sum_{p\neq q} S_\mu^{p\bar p}S_\mu^{q\bar q} D_1 g_{p\bar p} D_{\bar 1} g_{q\bar q} \\
 &\; + \sum_{p\neq q} (S_\mu^{p\bar p}S_\mu^{q\bar q}
 - S_\mu S_\mu^{p\bar p,q\bar q} ) D_1 g_{p\bar p}D_{\bar 1} g_{q\bar q}\Big).\\
 \end{aligned}
\end{equation}

We now estimate $P_m A_1^2$ case by case. For $\mu=1$,
it is easy to see
\[\begin{aligned}
\big(1+\frac{\beta}{2}\big) |\sum_j D_1 g_{j\bar j}|^2
\geq &\;  \big(1+\frac{\beta}{2}\big) \Big[ 2 \operatorname{Re} \sum_{a\neq 1} D_1 g_{a\bar a} \overline{ D_1 g_{1\bar 1} }
 +  | D_1 g_{1\bar 1} |^2 \Big]\\
 \geq &\; \big(1+\frac{\beta}{4}\big)  | D_1 g_{1\bar 1} |^2 - C_\beta \sum_{a\neq 1} | D_1 g_{a\bar a} |^2.
\end{aligned}\]
Then, using the fact $Q \geq \lambda_1 Q^{1\bar 1}$ for $\lambda \in \Gamma_{k+1}$,
we can derive from the first inequality of \eqref{ABCDE-1} that
\begin{equation}
\label{mu1}
\begin{aligned}
P_m^2 A_1\geq &\; \big(1+\frac{\beta}{4}\big) \frac{P_m \kappa_1^{m-1} Q}{S_1^2}   | D_1 g_{1\bar 1} |^2
- C_\beta  \frac{P_m \kappa_1^{m-1} Q}{S_1^2}  \sum_{a\neq 1} | D_1 g_{a\bar a} |^2\\
\geq &\; \frac{ \big(1+\frac{\beta}{4}\big) P_m \kappa_1^{m-2} Q^{1\bar 1}}{(1+ \sum_{j\neq 1}\lambda_j/\lambda_1 + \beta_0^1/ \lambda_1)^2}
| D_1 g_{1\bar 1}|^2
- C_\beta  \frac{P_m \kappa_1^{m-1} Q}{S_1^2}  \sum_{a\neq 1} | D_1 g_{a\bar a} |^2\\
\geq &\; P_m \kappa_1^{m-2} Q^{1\bar 1} | D_1 g_{1\bar 1} |^2
- C_\beta  \frac{P_m \kappa_1^{m-1} Q}{S_1^2}  \sum_{a\neq 1} | D_1 g_{a\bar a} |^2,
\end{aligned}
\end{equation}
where in the last inequality we
used $1+\frac{\beta}{4} \geq (1+ (n-1)\delta' + \beta_0^1 / \lambda_1)^2$ for sufficiently small positive $\delta'$
and sufficiently large $\lambda_1$.
For $\mu \geq 2$,
we have from Lemma \ref{GRW} with $\theta = 1$ that
\begin{equation}
\label{ABCDE-1'}
\begin{aligned}
P_m^2 A_1 \geq &\;  \frac{P_m \kappa_1^{m-1} Q}{S_\mu^2}  \Big( \sum_j|S_\mu^{j\bar j} D_1 g_{j\bar j}|^2
-\sum_{p\neq q} |F^{pq} D_1 g_{p\bar p} D_{\bar 1} g_{q\bar q}|\Big),
\end{aligned}
\end{equation}
where we used the notation $F^{pq} = S_\mu^{p\bar p} S_\mu^{q\bar q}- S_\mu S_\mu^{p\bar p,q\bar q}$.

We compute $F^{pq}$. Recall $S_\mu = \sigma_\mu + \sum_{s=0}^{\mu - 1} \beta_s^{\mu} \sigma_s$,
where $\beta^\mu_s \geq 0$ are constants depending on $s$ and $\mu$.
As in Section 2, we can similarly define $S_\mu^{a\bar a}$ and $S_\mu^{a\bar a, b\bar b}$.
Define
\[
S_{\mu-1} (\lambda | a) := \sigma_{\mu-1} (\lambda | a) + \sum_{s=0}^{\mu-1} \beta_s^\mu \sigma_{s-1} (\lambda | a),
\]
and similarly define $S_{\mu-1} (\lambda | ab)$ and $S_{\mu - 2} (\lambda | ab)$,
where we use notation that $\sigma_s = 0$ for $s < 0$ and $\sigma_0 = 1$.
Note that $S_{\mu-1} (\lambda | a) = S_\mu^{a\bar a}$ and $S_{\mu-2} (\lambda | ab) = S_\mu^{a\bar a, b\bar b}$.
By Lemma \ref{sigmak} (1), it is easy to show
\[\begin{aligned}
F^{pq}
= &\, S_{\mu- 1} (\lambda | pq)^2 - S_\mu (\lambda | pq) S_{\mu -2} (\lambda | pq).
\end{aligned}\]

We now estimate $F^{pq}$.
By the first result of Lemma \ref{sigmak-1}, we see
\begin{equation}
\label{S}
S_\mu^{a\bar a} \geq \frac{\lambda_1 \cdots \lambda_\mu }{\lambda_a} \; \mbox{if} \; a \leq \mu \; \mbox{and} \;
S_\mu^{a\bar a} \geq \lambda_1 \cdots \lambda_{\mu - 1} \; \mbox{if} \; a > \mu.
\end{equation}
For $a \leq \mu$ and $b \leq \mu$, by Lemma \ref{sigmak} (6),
we have
\[ \sigma_s (\lambda | ab) \leq C \frac{\lambda_1 \cdots \lambda_\mu}{\lambda_a \lambda_b},\;
\sigma_{\mu - 1} (\lambda | ab) \leq C \frac{\lambda_1 \cdots \lambda_{\mu+1} }{\lambda_a \lambda_b},\;
\sigma_{\mu} (\lambda | ab) \leq C \frac{\lambda_1 \cdots \lambda_{\mu+2} }{\lambda_a \lambda_b},\]
where $s\leq \mu - 2$.
Hence,
 we have
\[\begin{aligned}
S_{\mu - 1}(\lambda | ab) \leq  C \frac{1 + \lambda_{\mu + 1}}{\lambda_b} S_\mu^{a \bar a},\;\;
S_{\mu -  2} (\lambda | ab) \leq C \frac{\lambda_1 \cdots \lambda_{\mu}}{\lambda_a \lambda_b} \leq \frac{C}{\lambda_b} S_\mu^{a \bar a},
\end{aligned}\]
and
\[
S_{\mu} (\lambda | ab) \leq C \frac{1 + \lambda_{\mu + 1} + \lambda_{\mu + 1} \lambda_{\mu + 2} }{\lambda_b} S_\mu^{a \bar a}.
\]
Now, we can estimate that
\[\begin{aligned}
& \sum_{p,q\leq \mu} | F^{pq} D_1 g_{p\bar p} D_1 g_{q\bar q}|\\
\leq&\; C  \Big( \frac{(1+ \lambda_{\mu + 1})^2}{ \delta^2 \lambda_1^2}
+ \frac{ 1 + \lambda_{\mu + 1} + \lambda_{\mu + 1} \lambda_{\mu + 2} }{\delta^2 \lambda_1^2} \Big)
 \sum_{p\leq \mu} |S_\mu^{p\bar p} D_1 g_{p\bar p}|^2\\
 \leq &\; \frac{C}{\delta^2} \frac{(1 + \delta'\lambda_1 )^2 + 1 + \delta'\lambda_1 + \delta'^2 \lambda_1^2}{\lambda_1^2}
 \sum_{p\leq \mu} |S_\mu^{p\bar p} D_1 g_{p\bar p}|^2.
\end{aligned}\]
For any undetermined positive constant $\tau$, we obtain
\begin{equation}\label{Fpq-1}
\sum_{p,q\leq \mu} | F^{pq} D_1 g_{p\bar p} D_1 g_{q\bar q}| \leq \tau \sum_{p\leq \mu} |S_\mu^{p\bar p} D_1 g_{p\bar p}|^2
\end{equation}
as long as
$\delta' \leq \frac{\tau \delta^2}{8C} \; \mbox{and} \; \frac{1}{\lambda_1} \leq \delta \sqrt{\frac{\tau}{8C}}$.

For $a\leq \mu $ and $b > \mu$,
by \eqref{S},
we have that
\[
S_{\mu - 2} (\lambda | ab) \leq C \lambda_1 \cdots \lambda_{\mu - 2} \leq \frac{C}{\lambda_{\mu - 1}} S_\mu^{b\bar b}
\]
and
\[\begin{aligned}
S_{\mu - 1} (\lambda | ab) \leq C \frac{\lambda_1 \cdots \lambda_\mu}{\lambda_a} \leq C S_\mu^{a\bar a}
 \leq C S_\mu^{b\bar b}.\\
\end{aligned}\]
Note that $\sigma_s (\lambda | ab) \leq C \frac{\lambda_1 \cdots \lambda_\mu}{\lambda_a}$ for $s \leq \mu - 1$,
$a\leq \mu $ and $b > \mu$.
So we have
\[
S_\mu (\lambda | ab) \leq
C \Big( \frac{\lambda_1 \cdots \lambda_{\mu} }{\lambda_a} + \frac{\lambda_1 \cdots \lambda_{\mu + 1} }{\lambda_a} \Big)
\leq C (1 + \lambda_{\mu + 1} ) S_{\mu}^{a\bar a}.
\]
By the above calculations,
we have, for any positive constant $\tau$,
\begin{equation}\label{Fpq-2}\begin{aligned}
& \sum_{p\leq \mu, q > \mu} | F^{pq} D_1 g_{p\bar p} D_1 g_{q\bar q}|\\
\leq &\; \Big( \sum_{p\leq \mu, q > \mu}   S_{\mu- 1} (\lambda | pq)^2 + S_\mu (\lambda | pq) S_{\mu -2} (\lambda | pq) \Big)
 |D_1 g_{p\bar p}|| D_1 g_{q\bar q}|\\
\leq &\;  \tau \sum_{p\leq \mu} |S_\mu^{p\bar p} D_1 g_{p\bar p}|^2 + C_\tau \sum_{q> \mu} |S_\mu^{q\bar q} D_1 g_{q\bar q}|^2.
\end{aligned}\end{equation}

For $a > \mu$ and $b > \mu$, we have
\[\begin{aligned}
S_{\mu - i} (\lambda | ab) \leq  C \lambda_1 \cdots \lambda_{\mu - i},\; \mbox{where} \; i = 0, 1, 2.
\end{aligned}\]
So by \eqref{S} we obtain
\begin{equation}\label{Fpq-3}\begin{aligned}
& \sum_{p\neq q, p, q > \mu}  | F^{pq} D_1 g_{p\bar p} D_1 g_{q\bar q}| \leq C \sum_{q> \mu} |S_\mu^{q\bar q} D_1 g_{q\bar q}|^2.
\end{aligned}\end{equation}
Combining \eqref{Fpq-1}, \eqref{Fpq-2}and \eqref{Fpq-3}, we obtain for any positive constant $\tau$:
\begin{equation}
\label{Fpq-4}
\begin{aligned}
\sum_{p\neq q} |F^{pq} D_1 g_{p\bar p} D_1 g_{q\bar q}|
\leq &\;  2 \tau \sum_{p\leq \mu} | S_\mu^{p\bar p} D_1 g_{p\bar p}|^2 + C_\tau \sum_{q> \mu} |S_\mu^{q\bar q} D_1 g_{q\bar q}|^2.
\end{aligned}
\end{equation}

Substituting \eqref{Fpq-4} into \eqref{ABCDE-1'}, we have
\begin{equation}
\label{ABCDE-2}
\begin{aligned}
P_m^2 A_1
\geq &\;  \frac{P_m \kappa_1^{m-1} Q}{S_\mu^2}  \Big( (1-2\tau)|S_\mu^{1\bar 1} D_1 g_{1\bar 1}|^2
-C_\tau \sum_{q> \mu} |S_\mu^{q\bar q} D_1 g_{q\bar q}|^2\Big).
\end{aligned}
\end{equation}
Since $S_\mu \geq \lambda_1 \cdots \lambda_\mu$,
we see
\[
\frac{\kappa_1 S_\mu^{1\bar 1} }{S_\mu} \geq \frac{S_\mu - S_\mu (\lambda | 1)}{S_\mu}
\geq 1 - \frac{C \lambda_{\mu + 1}}{\lambda_1} - \frac{C}{\lambda_1},
\]
where we used $\sigma_\mu (\lambda | 1) \leq C \lambda_2 \cdots \lambda_{\mu + 1}$ and
$\sigma_s (\lambda | 1) \leq C \lambda_2 \cdots \lambda_\mu$  for $1 \leq s \leq \mu - 1$,
in the second inequality.
We can estimate the first term on the right hand side of \eqref{ABCDE-2} as below
\begin{equation}
\label{A-2}
\begin{aligned}
 &\;\frac{P_m \kappa_1^{m-1} Q}{S_\mu^2} (1-2\tau)|S_\mu^{1\bar 1} D_1 g_{1\bar 1}|^2 \\
 = &\; (1-2\tau) P_m \kappa_1^{m-2} \frac{Q}{\kappa_1} \Big(\frac{\kappa_1 S_\mu^{1\bar 1}}{S_\mu}\Big)^2
 |D_1 g_{1\bar 1}|^2\\
 \geq &\; (1-2\tau) (1 + \delta^m) \kappa_1^{2m-2} \frac{Q}{\kappa_1}
 \Big(1- C\frac{\lambda_{\mu+1}}{\lambda_1} - \frac{C}{\lambda_1} \Big)^2 |D_1 g_{1\bar 1}|^2\\
 \geq &\; (1-2\tau) (1-C\delta' - \frac{C}{\lambda_1})^2 (1 + \delta^m) \kappa_1^{2m-2} \frac{\lambda_1 Q^{1\bar 1}}{\kappa_1}
 |D_1 g_{1\bar 1}|^2,
\end{aligned}
\end{equation}
where in the last inequality we used $Q \geq \lambda_1 Q^{1\bar 1}$.
For $\delta'$ and $\tau$ small enough and $\lambda_1$ large enough, we obtain
\begin{equation}
\label{mu2}
\begin{aligned}
 P_m^2 A_1
 \geq  \kappa_1^{2m-2} Q^{1\bar 1} |D_1 g_{1\bar 1}|^2 - C_\tau \frac{P_m \kappa_1^{m-1} Q}{S_\mu^2}
 \sum_{q> \mu} |S_\mu^{q\bar q} D_1 g_{q\bar q}|^2.
\end{aligned}
\end{equation}

Substituting \eqref{mu1} or \eqref{mu2} into the inequality \eqref{ABCDE'}, we obtain
 \begin{equation}
\label{ABCDE4}\begin{aligned}
P_m^2 &\; (A_1 + B_1 + C_1 + D_1 - E_1) \\
\geq &\; P_m  \kappa_1^{m-2}\sum_{j>\mu} \Big( Q^{j\bar j}
- C_\tau \frac{\kappa_1 Q (S_\mu^{j\bar j})^2}{S_\mu^2} \Big)| D_1 g_{j\bar j}|^2.
\end{aligned}\end{equation}
Assume $\kappa_1 \leq 2\lambda_1$.
We show that
$Q^{j\bar j} - C_\tau \frac{\lambda_1 Q (S_\mu^{j\bar j})^2}{S_\mu^2} \geq 0 \;\mbox{for}\; j > \mu$.
For any $j>\mu$, since
$ S_\mu^{j \bar j} \leq C \lambda_1 \cdots \lambda_{\mu -1}$ and $S_\mu \geq \lambda_1 \cdots \lambda_\mu$,
we have
\[
C \frac{\lambda_1 (S_\mu^{j\bar j})^2}{S_\mu^2}
\leq C \frac{\lambda_1}{\lambda_\mu^2}
\leq \frac{C}{\lambda_1 \delta^2}.
\]
By the remark below Lemma \ref{lemm-2},
for $\mu<j \leq k$, we see
\[\begin{aligned}
Q^{j\bar j} = &\; \sum_{j < s \leq k} \alpha_s \sigma_s^{j\bar j} + \sum_{\mu < s \leq j } \alpha_s \sigma_s^{j\bar j}\\
\geq &\; \sum_{ j < s \leq k} \alpha_s \frac{\lambda_1\cdots\lambda_{s}}{\lambda_j} + \sum_{\mu < s \leq j } \alpha_s \lambda_1\cdots\lambda_{s-1}\\
\geq &\; \frac{1}{C \lambda_j} \sum_{ j < s \leq k} \alpha_s \sigma_s + \frac{1}{C \lambda_s} \sum_{\mu < s \leq j } \alpha_s \sigma_s \\
\geq &\; \frac{1}{C \delta' \lambda_1} \sum_{ \mu < s \leq k} \alpha_s \sigma_s
= \frac{Q}{C \delta' \lambda_1},
\end{aligned}\]
and
for $j>k$,
\[\begin{aligned}
Q^{j\bar j} = &\; \sum_{\mu < s \leq k } \alpha_s \sigma_s^{j\bar j}
 \geq \sum_{\mu < s \leq k } \alpha_s \lambda_1\cdots\lambda_{s-1} \\
 \geq &\; \sum_{\mu < s \leq k } \alpha_s \frac{\lambda_1\cdots\lambda_{s-1} \lambda_s }{\lambda_s}
 \geq \frac{1}{C \delta' \lambda_1}  \sum_{\mu < s \leq k } \alpha_s \sigma_s \geq \frac{ Q }{C \delta' \lambda_1}.
\end{aligned}\]
It follows that for $\delta'$ sufficiently small, we have
\[
Q^{j\bar j} \geq \frac{Q}{C \delta' \lambda_1} \geq \frac{C_\tau Q}{\lambda_1 \delta^2}
\geq C_\tau \frac{\lambda_1 Q (S_\mu^{j\bar j})^2}{S_\mu^2}.
\]
It follows from \eqref{ABCDE4} that
\[
P_m^2 (A_1 + B_1 + C_1 + D_1 - E_1) \geq 0.
\]
\end{proof}

With Lemma \ref{lemm-1} and Lemma \ref{lemm-2}, we prove that we may assume in \eqref{main inequality 7}
the \textbf{claim} holds, i.e.
for sufficiently large $m$,
\begin{equation}
\label{ABCDE0}
A_{i}+B_{i}+C_{i}+D_{i}-E_{i} \geq 0, \quad \forall i=1, \ldots, n.
\end{equation}
\begin{proof}
Set $\delta_1 = 1$. If $\lambda_2 \leq \delta_2 \lambda_1$ for $\delta_2 $ small enough, then by Lemma \ref{lemm-2} we see
that \eqref{ABCDE0} holds. Otherwise $\lambda_2 \geq \delta_2 \lambda_1$.
 If $\lambda_3 \leq \delta_3 \lambda_1$ for $\delta_3 > 0$ small enough, then by Lemma \ref{lemm-2} we see
that \eqref{ABCDE0} holds. Otherwise, we have $\lambda_3 \geq \delta_3 \lambda_1$. Proceeding iteratively,
we may arrive at $\lambda_k \geq \delta_k \lambda_1$. But in this case, an upper bound for $\lambda_1$ follows directly
from the equation as
\[
C \geq \sigma_k (\lambda) \geq \lambda_1 \cdots \lambda_k \geq (\delta_k)^{k-1} \lambda_1^k.
\]
Therefore, we may assume \eqref{ABCDE0} in \eqref{main inequality 7}.
This proves the claim.
\end{proof}

\end{document}